\documentclass[12pt]{amsart}
\usepackage{amsmath,amsthm,epsfig,amssymb,amscd,amsfonts} 
\usepackage[all]{xy}
\usepackage{young}
\newtheorem{thm}[subsection]{Theorem}

\newtheorem{lem}[subsection]{Lemma}

\newtheorem{rem}[subsection]{Remark}
\theoremstyle{definition}
\newtheorem{Def}[subsection]{Definition}
\newtheorem{Not}[subsection]{Notation}

\newtheorem{proposition-definition}[subsection]{Proposition-Definition}

\begin{normalsize} \end{normalsize}

\newcommand{\Pa}{{\mathcal P}}
\newcommand{\LR}{{\mathcal L}}

\newcommand{\ZZ}{{\mathbb Z}}
\newcommand{\QQ}{{\mathbb Q}}
\newcommand{\PP}{{\mathbb P}}
\newcommand{\NN}{{\mathbb N}}

\newcommand{\OOO}{{\mathcal O}}

\newcommand{\AAA}{{\mathcal A}}

\numberwithin{equation}{section}

\author{F. Laytimi}

\address {F. L.: Math\'ematiques - b\^{a}t. M2, Universit\'e Lille 1,
F-59655 Villeneuve d'Ascq Cedex, France}

\email {fatima.laytimi@univ-lille1.fr}

\author{W. Nahm}

\address{W. N.: Dublin Institute for Advanced Studies,
10 Burlington Road, Dublin 4, Ireland}

\email{wnahm@stp.dias.ie}


\title{Positivity of vector bundles and dominance}

\begin{document}

\date{}

\begin{abstract} 
Let $E$ be a vector bundle and $S_a$, $S_b$ the Schur functors associated to partitions $a$ and $b$. Previously we have shown that ampleness of $S_aE$ implies ampleness of $S_bE$ when $a$ is greater than $b$ in the dominance partial order. Here we prove that this result generalizes to $k$-ample, semiample and nef vector bundles.
Our proof uses the common algebraic nature of these three properties and an investigation of the Littlewood-Richardson rules.
\end{abstract}

\maketitle

\section{Algebraic Positivities} \setcounter{page}{1}
Various forms of positivity for line bundles have been studied in great depth, but for vector bundles even basic issues remain underexplored, despite old foundational work (\cite{Hart}, \cite{Sommese}) and a more recent resurgence following \cite{Lazarsfeld}. Our own motivation comes from connectedness and vanishing theorems \cite{Laytimi}, but here we shall explore more basic properties.

We will consider vector bundles $E$ on a fixed projective manifold $X$. Since we refer to \cite{Sommese}, we work with complex numbers, though extensions to other fields of characteristic zero are straightforward. When not stated otherwise, we transfer positivity properties from line bundles to vector bundles by the procedure of Hartshorne 
in \cite{Hart}, who was inspired by Grothendieck's approach in EGA II. Accordingly, a positivity property of $E$ is defined as the corresponding property of the line bundle $\OOO_{\PP E}(1)$. The points of $\PP E$ are given by a point $x$ in $X$ and a hyperplane in the fibre $E_x$ of $E$ over $x$. We realize them as hyperplanes annihilated by $span(u)$, where $u$ is a vector in $E_x^*$. The fibre of $\OOO_{\PP E}(1)$ over such a point can be identified with $span(u)^*$. When $E$ is a line bundle $L$, then $\OOO_{\PP L}(1)$ has a natural identification with $L$, because $\PP L$ reduces to $X$ and $span(u)$ to $L_x^*$.

The most commonly encountered form of positivity is ampleness, but the much weaker semiampleness or the intermediate form of $k$-ampleness in the sense of Sommese \cite{Sommese} are also of considerable interest. Recall that a line bundle $L$ over $X$ is semiample when $L^n$ is spanned by global sections for some positive integer $n$. Such a semiample bundle is $k$-ample if the fibers of the corresponding map 
$X\rightarrow \PP(H^0(X, L^n))$ have at most dimension $k$. Thus $k$-ampleness is equivalent to ampleness for $k=0$ and becomes weaker for increasing $k$ until it reduces to semiampleness for $k$ equal to the dimension of $X$. A vector bundle
$E$ is semiample when
$\OOO_{\PP E}(n)$ is spanned by global sections for some positive integer $n$, and so on. 

We shall generalize some properties of ampleness to such weaker forms of positivity. To avoid unnecessary repetition, we introduce a general concept.

\begin{Def}
Let $\AAA$ be a property of isomorphism classes of vector bundles on $X$. We write $\AAA(E)$ when $E$ has the property $\AAA$. We define algebraic properties $\AAA$ by the following three conditions 
\begin{description}
 \item [Additivity] $\AAA(E\oplus F)$ if and only if $\AAA(E)$ and $\AAA(F)$
 \item[Multiplicativity] if $\AAA(E)$ and $\AAA(F)$ then $\AAA(E\otimes F)$
 \item [Exponent elimination] if  $\AAA(E^{\otimes n})$ then $\AAA(E)$.
\end{description}
\end{Def}
Note that $\AAA(E)$ implies $\AAA(E^{\otimes n})$ by the multiplicativity condition. It is well known that ampleness is algebraic (\cite{Hart}, see \cite{Barton} for characteristic $p$).

\begin{thm}
 Semiampleness, $k$-ampleness and nefness are algebraic properties.
\end{thm}
\begin{proof}For exponent elimination, we make use of the inclusion $j_n: \PP E\rightarrow\PP E^{\otimes n}$ given by
$u\mapsto u^n$, for which 
$$j_n^*\OOO_{\PP E^{\otimes n}}(1) = \OOO_{\PP E}(n).$$
Exponent elimination for $\AAA$ is immediate when it is true for line bundles $L$ and when $\AAA(L)$ implies $\AAA(j^*L)$ for injections $j$, because in this case  $\AAA (\OOO_{\PP E^{\otimes n}}(1))$ implies $\AAA (\OOO_{\PP E}(1)^n)$ by using $j_n^*$. For $k$-ampleness this is stated in (\cite{Sommese}, Corollary 1.9). 

We give the argument for semiampleness. Semiampleness of a line bundle $L$ means that there is a positive power $L^l$ that is generated by sections. Thus semiampleness of $E^{\otimes n}$ means that
$\OOO_{\PP E^{\otimes n}}(l)$ and consequently its $j_n^*$ image $\OOO_{\PP E}(1)^{nl}$ are generated by sections. Thus $\OOO_{\PP E}(1)$
and by definition $E$ are semiample, too.
For nef vector bundles, the required properties are proven in (\cite{Lazarsfeld}, Thm. 6.2.12).

In \cite{ample} we have proven additivity and multiplicativity for semiampleness. For $k$-ampleness additivity and multiplicativity are given by Corollary 1.10 in \cite{Sommese}.
\end{proof}

Bigness is not an algebraic property, because it is not additive. Indeed, if $E, F$ are vector bundles over $X$ and $E$ is big, then $E\oplus F$ is big, without condition on $F$. This is an immediate consequence of Proposition 3.4 in  \cite{4}. 

Even for semiampleness one has to be careful about the way one generalizes from line bundles to vector bundles. Demailly's definition of strong semiampleness of $E$ states that some symmetric power $S^rE$ is generated by sections (\cite{Demailly}, 1.1). This property is not additive. Indeed, choose $X$ and a line bundle $L$ on $X$ so that $L$ has no section, but $L^2$ is the trivial bundle $I$. Then $L$ and of course $I$ are Demailly semiample, but $L\oplus I$ is not. Demailly semiampleness is a stronger assumption than what we use, because it states that for any non-vanishing $U$ in $S^rE_x^*$ there is a section $s$ of $S^rE$ for which $\langle s(x), U\rangle \neq 0$, whereas we require this only for $U$ of the form $u^r$ with $u\in E_x^*$.

As exemplified by Demailly's definition, using symmetric powers is as natural as using tensor powers. In our definition of algebraic properties, replacement of $E^{\otimes n}$ by $S^nE$ yields an equivalent result. In one direction, for characteristic zero, $S^nE$ is a direct summand of $E^{\otimes n}$. Thus when $\AAA$ is additive, $\AAA(E^{\otimes n})$ implies $\AAA(S^nE)$. Conversely, by multiplicativity, $\AAA(S^nE)$
implies $\AAA((S^nE)^{\otimes N})$ for any positive integer $N$. For suitable $N$ we will show as part of our main result Theorem \ref{Thm} that $\AAA((S^nE)^{\otimes N})$ implies $\AAA(E^{\otimes Nn})$.
Indeed, both $\AAA((S^nE)^{\otimes N})$ and $\AAA(E^{\otimes Nn})$ can be written as direct sums of terms that are isomorphic to $S_aE$ for a finite set of Schur functors $S_a$, with certain nonvanishing multiplicities. We shall prove that for suitable $N$ the two sets of Schur functors are exactly the same. Thus when $\AAA$ is additive, $\AAA((S^nE)^{\otimes N})$ and $\AAA(E^{\otimes Nn})$ are equivalent. 

We recall some basic facts. The Schur functors $S_a$ are labeled by integral partitions $a$. Let $E$ be a vector bundle of rank $d$ on $X$. 
Let $a=(a_1,\ldots, a_r)$ with $a_i\geq a_{i+1}$, $a_r>0$, be a partition of length $r$ with $r\leq d$. Then $S_aE$ is a nontrivial vector bundle on $X$. Recall that the weight of $a$ is given by
$$|a|=\sum_{k=1}^r a_i,$$ 
that $S_a E = S^{|a|}E$ for $r=1$ and $S_aE = \Lambda^{|a|}E$ for $a_1=1$. More generally,  $S_aE$ is isomorphic to a direct summand of $E^{\otimes|a|}$ for any $a$.

For ease of notation, let $a_k=0$ for $k>r$. Partitions are partially ordered by dominance, a relation we write as $\succsim$. 
For $a, b$ of equal weight, $a \succsim b$ means that
$$\sum_{k=1}^l a_i \geq \sum_{k=1}^l b_i$$
for $l=1,\ldots,d$. For arbitrary $a, b$ one can compare $|b|a$ and $|a|b$, because both partitions have the weight $|a||b|$. When
$|b|a \succsim |a|b$ and $|a|b \succsim |b|a$
then $|b|a = |a|b$, so that $a$ and $b$ are proportional, for example $b=na$.

When $|b|a \succsim |a|b$ and $S_a E$ is ample, then $S_bE$ is ample, as we proved in \cite{ample}. Here our main result is a generalisation of this fact to any algebraic positivity.

\begin{thm}
 Let $a,b$ be partitions so that $|b|a\succsim |a|b$. When $\AAA$ is an algebraic property, then $\AAA(S_a E)$ implies
 $\AAA(S_b E)$. In particular, for positive integers $n$, $\AAA(S_{na} E)$ is equivalent to $\AAA(S_a E)$.
\end{thm}

For the proof of this theorem we will use Theorem \ref{Thm}, which does not depend on this introduction. Let $M, N$ be positive integers so that $M|a|b = N|b|a$. By multiplicativity and additivity, the algebraic property $\AAA(S_a E)$ carries over to direct summands of $(S_a E)^{\otimes N}$. By Theorem \ref{Thm}, for suitable $N$ any direct summand of $(S_b E)^{\otimes M}$ is isomorphic to a direct summand of $(S_a E)^{\otimes N}$. By additivity and exponent elimination this establishes $\AAA(S_b E)$.

Our proof for the special property of ampleness given in \cite{ample} was based on properties of the Littlewood-Richardson rules, supplemented by a cohomological criterion. Our new proof is based on the Littlewood-Richardson rules only, but needs a slightly deeper investigation of their structure. On the other hand, it also covers the case of ampleness. Theorem \ref{Thm} will be proved in the next section.

Some readers may be interested in combinatorics but not geometry. Accordingly, section 2 is self-contained and does not refer to section 1.

\section{Partitions and their algebra} \setcounter{page}{1}
Let $\NN=\{1,2,\ldots\}$,  $\NN_0=\{0,1,2,\ldots\}$. Specific conventions for this section are the following. The letters $i,j,k,m,n,r,s,t,u$ will refer to elements of $\NN_0$. In contexts where they only refer to elements of $\NN$ this will be stated explicitely. The letters $\alpha,\beta,\gamma,\delta$ will refer to elements of $\NN$, except for Kronecker's $\delta$, which is defined by $\delta_{mn}=1$ for $m=n$ and $\delta_{mn}=0$ otherwise. The letter $l$ (for 'length') refers to elements of $\NN$. We put  $L(l)={\mathit lcm}\{1,\ldots,l\}$. We denote set differences by $\setminus$ and the concatenation with a finite sequence by $\circ$, like $(2,3)\circ(2,\ldots)= (2,3,2,\ldots)$.
 For one-element sets $\{(\alpha,\beta)\}$ in $\NN\times\NN$ we put $\rho(\{(\alpha,\beta)\}) =\alpha$, where $\rho$ stands for 'row'. The letter $I$ (for 'interval') stands for non-empty subseries of $\NN$. When $I=I_1\circ\ldots\circ I_\sigma$, we call $J=\{I_1,\ldots,I_\iota\}$ a subdivision of $I$. The letter $\iota$ will always refer to the number of terms in a given subdivision $J$.
For $\alpha\in I$ we define $J(\alpha)$ by $\alpha\in I_{J(\alpha)}$.

A partition (of infinite length) is a weakly decreasing series of numbers in $\NN_0$ that converges to zero. Its strict length is the number of terms in $\NN$. The set of partitions will be denoted by $\Pa$. Generic partitions will be denoted by capital letters, their elements by the corresponding lowercase letters, like $A=(a_1,a_2,\ldots)$ or $A^\sim=(\tilde a_1, \tilde a_2,\ldots,..)$. The letters $A,B,C,P$ will always refer to partitions. $\Pa$ forms a semigroup under standard addition, which will be denoted by $+$. Addition of several terms will be written with $\Sigma$ as usual, addition of $n$ equal terms $A$ yields $nA$. The Young tableau of $A$ is the set $Y(A)=\{(i,j)\in \NN\times \NN\ \vert\ 1\leq j\leq a_i\}$. In view of $Y(A)$ we call $a_i$ the length of row $i$ of $A$. Transposition in $\NN\times \NN$ will be denoted as in $(i,j)^\sim = (j,i)$. We define the partition $A^\sim$ by $Y(A^\sim)=Y(A)^\sim$. The length of column $i$ of $A$ is $\tilde a_i$.

Let $\Lambda^r=(r,0,\ldots)^\sim$. We have
\begin{equation}\label{+}
 A=\sum_{i=1}^\infty (\tilde a_i,0,\ldots)^\sim,
\end{equation}
so that $\Pa$ is generated by $\Lambda^1,\Lambda^2,\ldots$ as a semigroup.
The weight $|A|$ of $A$ is the cardinality $|Y(A)|$. The distance between two partitions $A,B$ of equal weight is defined as $|Y(A)\setminus Y(B)|$. We write $|A\setminus B| =|Y(A)\setminus Y(B)|$ and $\rho(A\setminus B) =\rho(Y(A)\setminus Y(B))$ for $|A\setminus B|=1$.

For $l\in\NN$, weakly decreasing series of $l$ numbers in $\NN_0$ are called partitions of length $l$. They form a semigroup $\Pa(l)$ with a natural injection into $\Pa$. Usually, we do not mention the length of partitions when it is arbitrary apart from constraints by the immediate context.
Among these constraints, expressions like $A+B$ are only defined for two partitions of the same length.

We denote the restriction of a partition $A$ to $I$ by $A[I]$, so that $A[I_1\circ I_2]=A[I_1]\circ A[I_2]$. For $l=|I|$ there is a natural bijective map from $\Pa(l)$ to weakly decreasing functions from $I$ to $\NN_0$. For $A\in\Pa(l)$ we denote the corresponding image by $A(I)$.

\begin{Def}
We define the dominance partial order by
$A\succ B$ iff $|A|=|B|$, $A\neq B$ and $|A[(1,\ldots,r)]|\geq|B[(1,\ldots,r)]|$ for all $r\in \NN$. 
\end{Def}
 
Since $|A[(1,\ldots,r)]|=|A|$ when $r$ is larger than the strict length of $A$, any statement $A\succ B$ reduces to a finite set of inequalities. We have $\Lambda^{|A|}\preccurlyeq A$.

\begin{lem}\label{Psequence}
  For $A\succ B$, $|A\setminus B|=k$, there is an interpolating sequence $P^0,P^1, \ldots,P^k$ of partitions of equal weight, with distance 1 between adjacent terms and $ P^0=A, P^k=B$. For $0\leq i<j\leq k$ one has  $Y(A^i)\setminus Y(A^j)\subset Y(A)\setminus Y(B)$ and $Y(A^j)\setminus Y(A^i)\subset Y(B)\setminus Y(A)$.
\end{lem}

\begin{proof}

 For $k=1$ the statement is trivial. We use induction on $k$.
Let $\alpha$ be the smallest integer so that $a_\alpha<b_\alpha$. Let $\beta$ be the largest interger in $\{1,\ldots,\alpha\}$ so that $a_\beta>b_\beta$. Define $A^1$ by  $a^1_\beta=a^0_\beta-1$, $a^1_\alpha=a^0_\alpha+1$ and $a^1_\gamma=a^0_\gamma$ for $\gamma$ different from $\alpha,\beta$. Then $A^0\succ A^1\succ B$, $|A^1\setminus B|=k-1$ and we can use the induction hypothesis. For the final statement, note that $Y(A^1)\setminus Y(B)\subset Y(A)\setminus Y(B)$ and $Y(B)\setminus Y(A^1)\subset Y(B)\setminus Y(A)$. 
 \end{proof}
We need some well-known facts about the  Littlewood-Richardson algebra $\LR$, which we collect here for reference and to introduce our notation. 
$\LR$ is a commutative $\ZZ$-algebra, with $\Pa$ as a free basis
and $\Lambda^0$ as unit. The partitions $\Lambda^r$, $r\in \NN$, generate $\LR$ as an algebra. By abuse of language, the injection of $\Pa$ into $\LR$ will not be written down.
\begin{Not}
 Addition and multiplication in $\LR$ will be denoted by $\oplus,\otimes$. We write ${\mathcal M}\lhd{\mathcal N}$ iff $0\leq m_A\leq n_A$ for all $A\in\Pa$ when ${\mathcal M}=\oplus_{A\in\Pa}m_AA$, ${\mathcal N}=\oplus_{A\in\Pa}n_AA$.
 The integer $m_A$ is called the multiplicity of $A$ in ${\mathcal M}$. ${\mathcal M}\rhd{\mathcal N}$ means ${\mathcal N}\lhd{\mathcal M}$. We define a map $\bullet: \NN_0\times \LR\rightarrow \LR$ by $a\bullet{\mathcal M} = \oplus_{A\in\Pa}m'_A\ (a)\circ A$, where $m'_A=m_A$ iff $(a)\circ A\in\Pa$ and $m'_A=0$ otherwise.
 Sometimes we need to extend $+$ to an operation in $\LR$ that is distributive over $\oplus$. For this operation we write $\dotplus$, as in $A\dotplus(B\oplus C)=(A+B)\oplus(A+C)$.
 The $N$-fold sum of $A$ with respect to $\oplus$ could be denoted by $A^{\oplus N}$, but will not be used. The notation $NA$ refers to $+$, as stated above.
\end{Not}

The multiplication in $\LR$ is defined recursively by 
\begin{equation}\label{multALambda}
A\otimes\Lambda^r=(a_1+1)\circ(A_2\otimes\Lambda^{r-1})\oplus (a_1)\bullet(A_2\otimes\Lambda^r),\end{equation}
where $A\in \Pa$, $A=(a_1)\circ A_2$, $r>0$. Recall that $A\otimes\Lambda^0=A$.

Note that $A\otimes\Lambda^r$ is equal to $A+\Lambda^r$ up to the addition of terms that are smaller in the dominance partial order. Due to eq. (\ref{+}) this means that eq. (\ref{multALambda}) determines all products in $\LR$. 
The following facts are consequences of this definition of $\otimes$.

$A\otimes B\rhd C$ implies $|A|+|B|=|C|$, so that $\LR$ is graded by weight. 

For $r\leq s$ one has
\begin{equation}\label{mult2Lambda}
 \Lambda^r\otimes\Lambda^s=\oplus_{k=0}^r (\Lambda^{r-k}+\Lambda^{s+k}).\end{equation}
\begin{align}\label{mult+}
\nonumber {\mathit If}\ \   A_i\lhd B_i\otimes C_i\ \ {\mathit for}\ \  i=1,2,\ \ {\mathit then}\\
A_1+A_2\lhd(B_1+B_2)\otimes(C_1+C_2).
\end{align}
 The latter relation generalizes immediately to multiple sums. For $C_2=\Lambda^0$ one obtains the useful special case
 \begin{equation}\label{multinert}
 (A+B)\otimes C\rhd A\dotplus(B\otimes C).
 \end{equation}
 For $B=\Lambda^0$ this yields
 $$A\otimes C\rhd A+C$$
 and consequently
\begin{equation}\label{multSigma}
\Sigma_{i =1}^kA^i\lhd\otimes_{i =1}^kA^i.
\end{equation}

 When $A\otimes C\rhd (A+C)\oplus P$, then $P\prec (A+C)$. Some such partitions are obtained as follows.
 \begin{lem}\label{smaller} 
 Let $A^1\succ A^2$, $C^1\succ C^2$, with distances $|A^1\diagdown A^2|=1$ and $|C^1\diagdown C^2|=1$.
 Let $\rho(A^1\setminus A^2)\geq \rho(C^2\setminus C^1)$ and $\rho(C^1\setminus C^2)\geq \rho(A^2\setminus A^1)$. Then
  $$A^1\otimes C^1\rhd A^2+C^2.$$
 \end{lem}
 \begin{proof}

 By eq. (\ref{mult+}) it is sufficient to consider the two columns for which $A^1$ differs from $A^2$ and the two columns for which $C^1$ differs from $C^2$. Thus one must prove that
$$(\Lambda^{\alpha_1}+\Lambda^{\alpha_2-1})\otimes
(\Lambda^{\gamma_1}+\Lambda^{\gamma_2-1})\rhd 
(\Lambda^{\alpha_1-1}+\Lambda^{\alpha_2}+
 \Lambda^{\gamma_1-1}+\Lambda^{\gamma_2}),$$
 where $\alpha_1>\alpha_2$, $\gamma_1>\gamma_2$, $\alpha_1\geq \gamma_2$, $\gamma_1\geq\alpha_2$.
 By eq. (\ref{mult+}) this follows immediately from eq. (\ref{mult2Lambda}).  \end{proof}

By the definition of $\otimes$ one has $Y(A\otimes B)\supseteq Y(A)$. Thus for any $l\in \NN_0$, the partitions of strict length greater than $l$ span an ideal of $\LR$. 
We denote the corresponding quotient algebra by $\LR(l)$. 
The elements of $\Pa(l)$ form a free basis of $\LR(l)$. In $\LR(l)$ the images
of $\Lambda^r$ with $0<r\leq l$ will still be denoted by 
$\Lambda^r$. As an algebra, $\LR(l)$ is generated by these elements.

\begin{Not}
For the image of $\Lambda^l$ in $\Pa(l)$ we also use the notation
$D(l)$, where $D$ stands for 'determinant'. For subseries $I$ of $\NN$ we denote the corresponding sequence by $D[I]$.
\end{Not}

\begin{equation}\label{multD}
(D(l)+A)\otimes B =D(l)\dotplus(A\otimes B).
\end{equation}

The dominance partial order carries over to $\Pa(l)$. 
For $|A|=nl+s$ with $s<l$ one has
$A\succcurlyeq nD(l)+\Lambda^s$.

 Let $J=(I_1,\ldots)$ be a subdivision of $\{1,\ldots,l\}$ and let $B,C\in\Pa(l)$.
 When $A_i\lhd B[I_i]\otimes C[I_i]$ for $i=1,\ldots,\iota$, then 
 $A_1\circ\ldots\circ A_{\iota}\in\Pa(l)$ and 
\begin{equation}\label{multcirc}
A_1\circ\ldots\circ A_{\iota}\lhd B\otimes C.
\end{equation}
This generalizes immediately to multiple products.

\begin{Not}
 For $A\in\Pa(l)$ let 
 $$\chi(A)=(-a_l,\ldots,-a_1).$$
 When  $J$ is a subdivision of $\{1,\ldots,l\}$, let
 $\chi_J(A)=\chi(A[I_1])\circ\ldots\circ \chi(A[I_{\iota}]).$
\end{Not}

\begin{lem}\label{chi} When $A+\chi(B)$ is a partition, then
 $$(A+\chi(B))\otimes B\rhd A.$$
\end{lem}
\begin{proof}

 By eqs. (\ref{multD}) and (\ref{multinert}) one has 
$$b_1D(l)\dotplus(A+\chi(B))\otimes B \rhd A\dotplus (b_1D(l)+\chi(B))\otimes B.$$
By eqs. (\ref{+}) and (\ref{mult+}) it is sufficient to show that 
$$(D(l)+\chi(\Lambda^r))\otimes\Lambda^r\rhd D(l).$$
This follows from eq. (\ref{mult2Lambda}).  
\end{proof}

\begin{rem}
When $\chi$ is extended to a linear endomorphism of $\Pa(l)\otimes\ZZ$, it yields a symmetry of $\Pa(l)$ in the sense that
$$(mD+\chi(A))\otimes (nD+\chi(B))=(m+n)D+\chi(A\otimes B)$$
whenever $(mD+\chi(A))$ and $(nD+\chi(B))$ lie in $\Pa(l)$.
The reader may note that many of our statements are symmetric in this sense, when the corresponding replacements are made. In particular, when $A\succ B$, then
$mD+\chi(A)\prec mD+\chi(B)$.
\end{rem}

\begin{lem}\label{Atensorl}
 For any $A\in\Pa(l)$,\\ $|A|D(l) \lhd A^{\otimes l}$ and $|A|D(l)+\chi(A)\lhd A^{\otimes (l-1)}$.
\end{lem}

\begin{proof}

 By eqs. (\ref{+}) and (\ref{mult+}) it is sufficient to consider $A$ of the form $\Lambda^r$ with $1\leq r\leq l$. In this case, the result follows from iterating eq. (\ref{mult2Lambda}), keeping always the unique lowest term in the dominance partial order.
\end{proof}

After these generalities, we come to the specific calculations that we need for our proof.

\begin{lem}
 Let $r+s>0$, $t+u>0$, $r+t>0$, $s+u>0$, $r+t+s+u\leq l$. Then
\begin{equation}\label{exchange}
 (\Lambda^r+\Lambda^{l-t})\otimes (\Lambda^s+\Lambda^{l-u})\\
\rhd\ D(l) +\Lambda^{r+s-1}+\Lambda^{l-t-u+1}.
\end{equation}
\end{lem}
\begin{proof}

 For $rstu=0$ this follows immediately from eqs. (\ref{multALambda}) and (\ref{multD}). In this case the multiplicity of the partition on the right hand side in the product under investigation is $1$. Otherwise we use 
$\Lambda^s+\Lambda^{l-u}=(\Lambda^s\otimes\Lambda^{l-u}) \ominus(\Lambda^{s-1}\otimes \Lambda^{l-u+1})$. By (\ref{multALambda}), 
\begin{align}
\nonumber &(\Lambda^r+\Lambda^{l-t})\otimes\Lambda^s=\\
\nonumber &(\Lambda^{r+s-1}+\Lambda^{l-t}+\Lambda^1)
 \oplus(\Lambda^{r+s}+\Lambda^{l-t})
 \oplus(\Lambda^{r+s-1}+\Lambda^{l-t+1})\oplus {\mathcal R},
\end{align}
where $P\lhd{\mathcal R}$ implies that $|P[(1,\ldots,r+s-1)]|< 2(r+s-1)$ or that $p_{r+s-1}=1$. In these cases $P\otimes\Lambda^{l-u}$ cannot contribute to the multiplicity under investigation. Similarly,
$$(\Lambda^r+\Lambda^{l-t})\otimes\Lambda^{s-1}=
(\Lambda^{r+s-1}+\Lambda^{l-t})\oplus {\mathcal R}',$$
where ${\mathcal R}'$ has the same property. One finds that the wanted multiplicity is $3-1$, which is greater than 0.
\end{proof}

\begin{Def}\label{GH}
 Let $J$ be a subdivision of $\{1,\ldots,l\}$.
For $A\in \Pa(l)$ we define
$$G_J(A)=(L/|I_1|)|A[I_1]|D[I_1]\circ\ldots\circ (L/|I_{\iota}|)|A[I_{\iota}]|D[I_{\iota}],$$
where $L=L(l)$.
For $m,n\in\{1,\ldots,\iota\}$, let
$$H_J(A,m,n)=G_J(A)+E_J(m)+F_J(n),$$ where
\begin{align}
\nonumber E_J(m)&=(\delta_{m1}\chi(\Lambda^1(I_1)))\circ\ldots\circ(\delta_{m\iota}\chi(\Lambda^1(I_{\iota})))\\ 
\nonumber F_J(n)&=(\delta_{n1} \Lambda^1(I_1))\circ\ldots\circ(\delta_{n\iota}\Lambda^1(I_{\iota})).
\end{align}
\end{Def}

Note that $G_J(A)$ is a partition, whereas $H_J(A,m,n)$ need not be.

\begin{lem}\label{Gintensor}
\begin{align}
\nonumber G_A(J)&\lhd A^{\otimes L(l)}\\
\nonumber G_A(J)+\chi_J(A)&\lhd A^{\otimes (L(l)-1)}.
\end{align}
\end{lem}

Due to eq. (\ref{multcirc}) this is an immediate consequence of Lemma (\ref{Atensorl}). $\Box$

\begin{lem}\label{Hintensor}

 Let $J, A$ be as in the previous definition. Let $1\leq\delta<\beta$ with $\tilde a_\beta>0$ and $m=J(\tilde a_\beta)$, $n=J(\tilde a_\delta+1)$.
 Then $H_J(A,m,n)$ is a partition and 
 $$H_J(A,m,n)\lhd A^{\otimes L(l)}.$$
\end{lem}
\begin{proof}

 There is at least one line in $A[I_m]$, namely line $\tilde a_\beta$, that is longer than the lines in $A[I_{m+1}]$. Thus $|A[I_m]|/|I_m| > |A[I_{m+1}]|/|I_{m+1}|$. Analogously, for $n>1$ there is at least one line in $A[I_n]$, namely line $\tilde a_\delta+1$, that is shorter than the lines in $A[I_{n-1}]$. Thus $H_J(A,m,n)$ is a partition.

By eq. (\ref{mult+}) it is sufficient to show that 
$$H_J(A,m,n)\lhd (G_A(J)+\chi_J(A))\otimes A.$$
Whenever $A=B+C$, then 
\begin{align}
(G_A(J)+\chi_J(A))&=(G_B(J)+\chi_J(B))+(G_C(J)+\chi_J(C)),\\
H_J(A,m,n)&=H_J(B,m,n)+(G_C(J)+\chi_J(C)).
\end{align}
By eq. (\ref{mult+}) it remains to show that for suitable $B$
$$H_J(B,m,n)\lhd (G_B(J)+\chi_J(B))\otimes B.$$
The restriction of this relation to $I_r$ with $r<m$ or $r>n$ is trivially true, so that by eq. (\ref{multcirc}) we only need to consider the restriction to $I_m\circ\ldots\circ I_n$.
For $m=n$ we take $B=(\tilde a_\beta)^\sim$. Since
$\tilde a_\delta+1>\tilde a_\beta$, $B[I_m]=\Lambda^t(I_m)$ for some $t$ with $1\leq t<|I_m|$. We have to consider $\Lambda^{|I_m|-t}(I_m)\otimes \Lambda^t(I_m)$ and the claim can be verified by the selection of the term with $k=r-1$ in eq. (\ref{mult2Lambda}).
For $m<n$ we take
$B=(\tilde a_\delta)^\sim + (\tilde a_\beta)^\sim$
and restrict to $I_m\circ\ldots\circ I_n$. Then the wanted result is the case $r+s=|I_m|$, $t+u=|I_n|$, $l=|I_m\circ\ldots\circ I_n|$ of eq. (\ref{exchange}). 
\end{proof}

\begin{lem}\label{HmultP}
Let $J, A$ be as before. Let $H_J(A,m,n)$ be a partition.
 Let $P\succ P'$, $|P\setminus P'|=1$,
 $J(\rho(P\setminus P'))\leq m$ and $n\leq J(\rho(P'\setminus P))$, furthermore either $m<n$ or $m=n$ and $|I_m|>1$. Then $H_J(A,m,n)\otimes P\rhd G_J(A)+P'.$
\end{lem}
\begin{proof}

 One has $H_J(A,m,n)\succ G_J(A)$, $|H_J(A,m,n)\setminus G_J(A)|=1$ and can apply Lemma (\ref{smaller}). 
\end{proof}

\begin{lem}\label{AmultPP}
Let $J$ be a subdivision of $\{1,\ldots,l\}$ and $A,B\in \Pa(l)$. Let $A\succ B$, $|A\setminus B|=k$.
 Then $A^{\otimes kL(r)}\otimes A\rhd (kG_J(A)+B).$
\end{lem}
\begin{proof}

Let $P^0,\ldots,P^k$ be an interpolating sequence as in Lemma (\ref{Psequence}) with $P^0=A$ and $P^k=B$. We use induction on $k$. The case $k=0$ is trivial. Let
$Y(P^{k-1})\setminus Y(P^k)=\{(\alpha,\beta)\}$,\ \ $Y(P^k)\setminus Y(P^{k-1})=\{(\gamma,\delta)\}$, $m=J(\tilde a_\beta)$, $n=J(\tilde a_\delta+1)$. Because $\{(\alpha,\beta)\}\in Y(A)$ and $\{(\gamma,\delta)\}\notin Y(A)$ we have
$J(\alpha)\leq m$, $n\leq J(\gamma)$. Moreover, $\tilde a_\beta\leq \tilde a_\delta$, so that $m<n$ or $m=n$ and
$|I_m|\geq 2$.

By Lemma (\ref{Hintensor}), the induction hypothesis, eq. (\ref{multinert}) and Lemma (\ref{HmultP}) we have
\begin{align}
\nonumber  A^{\otimes kL(r)}\otimes A\ &\rhd H_J(A,m,n)\otimes A^{\otimes (k-1)L(r)}\otimes A\\
\nonumber &\rhd H_J(A,m,n)\otimes ((k-1)G_J(A)+P^{k-1})\\
\nonumber &\rhd (k-1)G_J(A)\dotplus (H_J(A,m,n)\otimes P^{k-1})\\
\nonumber &\rhd kG_J(A)+P^k.
\end{align}
\end{proof}

\begin{thm}\label{Thm}
 For any $A\in \Pa(l)$ there exists $N_A\in\NN$ so that for $n\geq N_A$ and any partition $B$ with $B\preccurlyeq nA$ one has $B\lhd A^{\otimes n}$.
\end{thm}

\begin{proof}

The set of partitions $B$ for which $B\preccurlyeq nA$ for some $n\in\NN_0$ forms a cone ${\mathcal C}(A)$ in $\NN_0^l$. By equation (\ref{multSigma}), the partitions $B$ for which $B\lhd A^n$ for some $n$ form a subsemigroup ${\mathcal C}'(A)$ of ${\mathcal C}(A)$. 
Let $\mathcal J$ be the set of all subdivisions of $\{1,\ldots,l\}$.
By Lemma (\ref{Gintensor}), $G_A(J)\in {\mathcal C}'(A)$ for any $J\in {\mathcal J}$. 
In (\cite{ample}, Lemma (3.11)) we have shown that ${\mathcal C}(A)\otimes \QQ_{\geq 0}$ is generated by the set of $G_A(J)$ with $J\in {\mathcal J}$. The cone ${\mathcal C}(A)$ can be written as a finite union of simplicial cones. These simplicial cones are labeled by subsets $S\in {\mathcal J}$ such that
the sequences $G_J(A)$ with $J\in S$ are linearly independent.
Any $B$ in such a simplicial cone has the form
$B=\sum_{J\in S} \kappa_J G_A(J)$, where the $\kappa_J$ are numbers in $\QQ_{\geq 0}$ with bounded denominators. We write $\kappa_J = \lfloor \kappa_J\rfloor +\lambda_J$.
The partitions $P$ of the form $\sum_{J\in S}\lambda_J G_A(J)$ form a finite subset $\Phi$ of ${\mathcal C}(A)$. To show that ${\mathcal C}(A)\setminus {\mathcal C}'(A)$ is a finite set, it suffices to show that for any $P\in\Phi$ and any $J\in {\mathcal J}$ there is an integer $N(P,J)$ such that $N(P,J)G_A(J)+P \in {\mathcal C}'(A)$. Indeed, then $\sum_{J\in\Sigma} \rho_J G_A(J)\in {\mathcal C}'(A)$ whenever $\rho_J \geq N(P,J)$ for at least one $J\in {\mathcal J}$.
Let $|P|=n_P|A|$, $|P\setminus n_PA|=k_P$.
Since $G_J(A)\preccurlyeq L(l)A$, one as $P\preccurlyeq n_PA$.
By Lemma (\ref{AmultPP}) one has
$(k_Pn_PG_J(A)+P)\lhd (n_PA)^{\otimes k_PL(r)}\otimes n_PA$, where we have used that $G_J(n_PA)=n_PG_J(A)$. By eq. (\ref{multSigma}), $nA\lhd A^{\otimes n}$. Thus choosing $N(P,J)=k_Pn_P$ is sufficient and the theorem is true with
$N_A=L(l)|{\mathcal J}|M$, where $M$ is the maximum of $(k_P+1)n_P$ over $P$ in $\Phi$. 
\end{proof}

 \begin{rem}
 Small examples and some arguments suggest that $N_A=l$ is sufficient for all $A$ in $\Pa(l)$. This is true for $l<4$, but we are not aware of a general proof. The statement of the theorem becomes wrong when one demands $N_A<l$. What happens for exponent $N=2$ is known but quite intricate.
For exponents between $2$ and $l$ it is easy to write down a natural interpolating statement, but we do not have enough experience in combinatorics to advance along this path. Thus we have used an approach that is blunt but straightforward.
\end{rem}

On behalf of all authors the corresponding authors state that there is no conflit of interrest

  \end{document}